\documentclass[12pt,leqno,fleqn]{amsart}  
\usepackage{amsmath,amstext,amsthm,amssymb,amsxtra}
\usepackage{txfonts} 
\usepackage[T1]{fontenc}
\usepackage{lmodern}

 \usepackage{euler}   

\usepackage{mathtools}
\mathtoolsset{showonlyrefs,showmanualtags}

\usepackage{hyperref} 
\hypersetup{
    colorlinks=true,       
    linkcolor=blue,          
    citecolor=magenta,        
    filecolor=magenta,      
    urlcolor=cyan           
}

\usepackage{amsrefs}  

\allowdisplaybreaks
\swapnumbers

\theoremstyle{plain} 
\newtheorem{lemma}[equation]{Lemma} 
 
\newtheorem{theorem}[equation]{Theorem}

\theoremstyle{definition}
\newtheorem{definition}[equation]{Definition} 

\theoremstyle{remark}
\newtheorem{remark}[equation]{Remark}

\newtheorem*{ack}{Acknowledgment}

\numberwithin{equation}{section}

\title[] {Weighted Weak Type Estimates for Square Functions} 
\author{Michael Lacey}   

\address{ School of Mathematics, Georgia Institute of Technology, Atlanta GA 30332, USA}
\email {lacey@math.gatech.edu}
\thanks{Research supported in part by grant NSF-DMS 0968499, 
and  a grant from the Simons Foundation (\#229596 to Michael Lacey). 
The first author benefited from the research program Operator Related Function Theory and Time-Frequency Analysis at the Centre for Advanced Study at the Norwegian Academy of Science and Letters in Oslo during 2012�2013.
}

\author{James Scurry}   

\address{ School of Mathematics, Georgia Institute of Technology, Atlanta GA 30332, USA}
\email {lacey@math.gatech.edu}

\begin{document}

	\maketitle

\begin{abstract}
For $ 1 < p < \infty $ and weight $ w \in A_p$, the following weak-type inequality holds for 
a Littlewood-Paley square function $ S$,  
	\begin{equation*}
 		\lVert S f\rVert_{ L ^{p, \infty } (w)} \lesssim
 		 [w] _{A_p} ^{\max \{ \frac 12 , \frac 1p \}} 
 		\phi([w]_{A_p})  \lVert f\rVert_{L ^{p} (w)} \,. 
 \end{equation*}
where $\phi_p(x) =1$ for $1 < p < 2$ and $\phi_p(x) = 1 + \log x$ for $2 \leq p $.
Up to the logarithmic term, these estimates are sharp.  
\end{abstract}
	
\section{Introduction} 

Our focus is on weak-type estimates for   square functions on weighted $ L ^{p}$ spaces, for Muckenhoupt $ A_p$ weights.  
Following M.~Wilson \cite{wilson1} define the intrinsic square function $G_{\alpha}$ as follows.  
\begin{definition}
Let $C_{\alpha}$ be the collection of functions $\gamma$ supported in the unit ball with mean zero and such that $| \gamma(x) - \gamma(y) | \le |x - y|^{\alpha}$. For $f \in L^1_{\rm{loc}}(\mathbb{R}^n)$ let 
\begin{align*}
A_{\alpha}f(x,t) &= \sup_{\gamma \in C_{\alpha}} |f \ast \gamma_t(x)|
\end{align*}
where $\gamma_t(x) = t^{-n} \gamma(x t^{-n})$ and take
\begin{align*}
G_{\alpha}f(x) &= \bigl( \int_{\Gamma(x)} A_{\alpha}f(y,t)^2   \frac {dy dt} {t ^{n+1}} \bigr)^{\frac{1}{2}}
\end{align*}
where $\Gamma(x):= \{ (y,t) \in \mathbb R ^{n+1}_+ \;:\; \lvert  y\rvert < t \}$ is the cone of aperture one in the upper-half plane.
\end{definition}

This square function  dominates many other square functions.   
Recall the definition of $ A_p$ weights.  

\begin{definition}\label{d:Ap} Let $ 1< p < \infty $. A weight $ w$ is in $ A_p$ if $ w $ has density $w (x)$, 
we have $ w (x) > 0 \textup{$ $ a.e.}$, and for $ \sigma (x) := w (x) ^{1-\frac p  {p-1}}$ there holds 
\begin{equation*}
[w] _{A_p} := \sup _{Q} \frac {w (Q)} {\lvert  Q\rvert } \bigl[
\frac {\sigma (Q)} {\lvert  Q\rvert }
\bigr] ^{p-1} < \infty 
\end{equation*}
where the supremum is formed over all cubes $ Q \subset \mathbb R ^{n}$.  
\end{definition}

The main result of this note is as follows. 

\begin{theorem} \label{t.intrinsic}
 For $1 < p < 3$, $ 0 < \alpha \le 1$,  and $w \in A_p$ the following inequality holds. 
\begin{gather} \label{e:G}
\lVert G _{\alpha } f\rVert_{L ^{p} (w) \to L ^{p, \infty } (w)} \lesssim [w] _{A_p} ^{\max \{ \frac 12 , \frac 1p \}}\phi([w]_{A_p}) \lVert f\rVert_{L ^{p} (w)}\,, 
		\\
\textup{where} \quad 
\phi([w]_{A_p}) := 
\begin{cases}
1 & 1 < p < 2 
 \\
 (1 + \log [w]_{A_p}) & 2\le p < 3 
\end{cases} 
 \end{gather}
\end{theorem}

By example, we will show that the power on $ [w] _{A_p}$, but not the logarithmic term, is sharp.  
This result can be contrasted with these known results.  First, for the maximal function $ M$, one has the 
familiar estimate of Buckley \cite{buckley}, 
\begin{equation*}
\lVert M \rVert_{L ^{p} (w) \to L ^{p, \infty } (w)} \lesssim [w] _{A_p} ^{1/p}\,, \qquad 1< p < \infty \,.  
\end{equation*}
Thus, the square function estimate equals that for $ M$ for $ 1< p < 2$, but is otherwise larger.  
There is also the recent sharp estimate of the strong type norm of $ G _{\alpha }$: 
\begin{align*}
\lVert G _{\alpha } \rVert_{L^p(w) \rightarrow L^{p}(w)} &\lesssim [w]_{A_p}^{\max \{ \frac 1 2, \frac{1}{p-1} \} }.
\end{align*}  
The weak-type estimate above is smaller for all values of $ 1 < p <3$, and is otherwise larger by the logarithmic term.  
The case of $ p=2$ in the dyadic strong-type inequality was proved by Wittwer
 \cite{wittwer}, also see  \cite{htv}.  The dyadic case, for general $ p$, was proved by 
Cruz-Uribe-Martell-Perez \cite{cump}, while the inequality as above is the main result of Lerner's paper \cite{lerner1}.

The case $p=1$ of Theorem~\ref{t.intrinsic} holds more generally.  Chanillo-Wheeden \cite{cw}, first for the area function, and 
Wilson \cites{wilson1,wilson2}, showed that for any weight $ w$, 
\begin{align} \label{e:1}
w \{  G _{\alpha }f > \lambda \} &\lesssim \frac{1}{\lambda} \int_{\mathbb{R}^n} \lvert  f \rvert  \cdot  Mw  \; dx 
\end{align}
where $M$ is the Hardy-Littlewood maximal function. In particular, \eqref{e:G} holds for $ p=1$. 

There are interesting points of comparison with the weak-type estimates for Calder\'on--Zygmund operators.  
Hyt\"onen \cite{2912709} established the strong type estimate.   For  $T$ an $L^2(\mathbb{R}^n)$ bounded Calderon-Zygmund operator, 
there holds 
\begin{align*}
 \lVert T \rVert_{L^p(w) \rightarrow L^{p}(w)} &\lesssim [w]_{A_p} ^{\max \{\frac 1 {p-1}, 1\}}\qquad 1< p < \infty \,. 
\end{align*}
Hyt\"onen et.~al. \cite{l7} the weak-type estimate 
\begin{align*}
 \lVert T \rVert_{L^p(w) \rightarrow L^{p,\infty}(w)} &\lesssim [w]_{A_p}, \qquad 1< p < \infty \,. 
\end{align*}
But, the $ L ^{1}$-endpoint variant of \eqref{e:1} fails, as was shown by Reguera \cite{MR2799801}, for the 
dyadic case and Reguera-Thiele \cite{1011.1767} for the continuous case.  
Specializing to the case where $ w \in A_1$, Lerner-Ombrosi-Perez \cite{2480568} have shown that 
 \begin{align*}
\lVert T \rVert_{L^1(w) \rightarrow L^{1,\infty}(w)} &\lesssim [w]_{A_1} (1+\log [w]_{A_1}). 
\end{align*}
And, in a very interesting twist, some power of the logarithm is necessary, by the argument of 
Nazarov et~al.~\cite{1/7}. 
It seems entirely plausible to us that in the case of $ p=2$ in \eqref{e:G}, that some power of the logarithm is required.  

\begin{ack}
The question of looking at the weak-type inequalities was suggested to us by  A.~Volberg \cite{voPC}, and we thank 
S.~Petermichl for conversations about the problem.  
\end{ack}

\section{Proof of Theorem \ref{t.intrinsic}} 

Our argument will apply the Lerner median inequality \cite{MR2721744}---a widely used technique, see \cites{cump,lerner1,landh}, among other  papers.  
To this  end, we need some definitions.  
For a  constant $ \rho > 0$, let us set $ \rho  Q$ to be the cube with the same center as $ Q$, and side length 
$ \lvert  \rho Q\rvert ^{1/n}= \rho \lvert  Q\rvert ^{1/n}  $.   For any cube  $ Q$, 
set $ \lvert  Q\rvert \langle f \rangle_Q := \int _{Q} f \; dx  $. 
We say that a collection of dyadic cubes $ \mathcal S$ is \emph{sparse} if  there holds 
\begin{equation*}
\Bigl\lvert  \bigcup \{ Q' \in \mathcal S \;:\; Q'\subsetneq Q\} \Bigr\rvert < \tfrac 12 \lvert  Q\rvert \,, 
\qquad Q\in \mathcal S\,. 
\end{equation*}  
For a sparse collection of cubes $ \mathcal S$ and $ \rho >1$ we define 
\begin{equation*}
\bigl(T _{\mathcal S, \rho } f  \bigr) ^2 := \sum_{Q\in \mathcal S} \langle f \rangle _{\rho Q} ^2 \mathbf 1_{Q} \,. 
\end{equation*}

Fix $ f$ supported on a dyadic cube $ Q_0$.  By application to Lerner's median inequality (compare to \cite{lerner1}*{(5.8)}), 
for $ N\to \infty $, there are constants $ m_N \to 0$ so that there is a sparse collection of cubes $ \mathcal S_N$ contained 
in $ N Q_0$ so that, for $ \rho = 45$, the following pointwise estimate holds.   
\begin{equation*}
\lvert  G _{\alpha }f (x) ^2 - m_N\rvert \cdot  \mathbf 1_{NQ_0} (x)
\lesssim M f (x) ^2 + T _{\mathcal S_N} f (x) ^2 \,.  
\end{equation*}
Therefore, in order to estimate the $ L ^{p,\infty } (w)$ norm of $ G _{\alpha } f$,  it suffices to estimate 
$ L ^{p} (w)$ norm of $ M f$ and of $T _{\mathcal S} f  $, for any sparse collection of cubes $ \mathcal S$. 

Now, by Buckley's bounds \cite{buckley},  $\lVert M  \rVert_{ L ^{p} \to L^{p,\infty}(w)} \lesssim [w]_{A_p}^{1/p}$.
As a result, to obtain Theorem $\ref{t.intrinsic}$ it suffices to show this Theorem.

\begin{theorem}\label{t:sparse} For $ 1< p < 3 $, weight $ w \in A_p$, 
any sparse collection of cubes $ \mathcal S$ and any $ \rho \ge 1$, there holds 
\begin{equation*}
\lVert T _{\mathcal S} \rVert_{ L ^{p} (w) \to L ^{p, \infty } (w) } \lesssim [w] _{A_p} \phi ([w ] _{A_p}) \,. 
\end{equation*}
\end{theorem}

We turn to the proof of this estimate.   With $ \rho >1$ fixed, it is clear that it suffices to consider 
collections $ \mathcal S$ which satisfy this strengthening of the defintion of sparseness: 
On the one hand, 
\begin{equation*}
\Bigl\lvert  \bigcup \{ Q' \in \mathcal S \;:\; Q'\subsetneq Q\} \Bigr\rvert < \tfrac  {\lvert  Q\rvert }{ 8 \rho ^{n}}  \,, 
\qquad Q\in \mathcal S\,. 
\end{equation*}  
and on the other, if $ Q \neq Q'\in \mathcal S $ and $ \lvert  Q\rvert = \lvert  Q\rvert $, then $ \rho Q \cap \rho Q'= \emptyset $. 
We can assume these conditions, as a sparse collection is the union of $ O (\rho ^{n+1})$ subcollections which meet these 
conditions, and we are not concerned with the effectiveness of our estimates in $ \rho $.

Let $\mathcal S^1$ consist of all $Q$ such that $\langle f \rangle_{\rho Q} >  1 $. Then if $Q \in \mathcal  S^1$ we have $ Q \subset \{ Mf >  1 \}$ so that 
\begin{align*}
 w \Bigl\{ \sum_{Q \in \mathcal S ^{1}} \langle f \rangle_{\rho Q}^2 \mathbf 1_{Q} > 1  
	\Bigr\} &\le w \Bigl\{
	\bigcup _{Q\in \mathcal S ^{1}} Q 
	\Bigr\} 
	\le w \{ M f > 1 \} \lesssim [w] _{A_p}   \lVert f\rVert_{L ^{p} (w)} ^{p}.
	\end{align*}

We split the remaining cubes into disjoint collections setting
\begin{equation*} 
	\mathcal S _{\ell } := \{Q \in \mathcal S \;:\; 2 ^{- \ell -1}  < \langle f \rangle _{\rho Q}  \le 2 ^{- \ell }  \}\,, \qquad \ell =0, 1 ,\dotsc, 
\end{equation*}
Now let $E(Q) = \rho Q \backslash R (Q)$ where  
and $R(Q) = \bigcup \{ \rho Q': Q' \subsetneq  \rho Q,\ Q' \in \mathcal S_{\ell} \}$. Notice,
that  $|R(Q)| <  \tfrac 18 \lvert  \rho Q\rvert $, whence 
\begin{align*}
\langle f \mathbf{1}_{E(Q)} \rangle_{\rho Q} &=
\langle f \rangle_{\rho Q} - \langle f \mathbf{1}_{R(Q)} \rangle_{\rho Q} \\
&\ge \langle f \rangle_{\rho Q }- 8^{-1} \langle f \rangle_{R(Q)} \\
&\ge 2^{-\ell-1}  - 8^{-1} 2^{-\ell} \gtrsim 2^{-\ell}   \,. 
\end{align*} 
That is, we have good lower bound on these averages, and moreover the sets $ E (Q)$ are pairwise disjoint in $ Q\in \mathcal S _{\ell }$. 
We will estimate 
\begin{equation} \label{e:E}
\sum_{Q\in \mathcal S _{\ell }} 2 ^{- 2 \ell } \mathbf 1_{Q} 
\lesssim \sum_{Q\in \mathcal S _{\ell }} \langle f \mathbf{1}_{E(Q)} \rangle_{\rho Q} ^2 \mathbf 1_{Q} \,. 
\end{equation}

The following lemma is elementary.  

\begin{lemma}\label{l.weakrho} Let $ \mathcal T$ be a  collection of cubes.  
	We have, for $ 1< p < \infty $, and sequences $ \{g _{Q} \;:\; Q\in \mathcal T\}$ of non-negative functions, 
\begin{equation*}
 \Bigl\lVert 
\Bigl[
\sum_{Q\in \mathcal T}  \langle g _{Q} \rangle_{\rho Q} ^p \mathbf{1}_{Q}  
\Bigr] ^{1/p} 
\Bigr\rVert_{L ^{p} (w)} \lesssim   [w] _{A_p} ^{1/p} 
	\Bigl\lVert 
\Bigl[
\sum_{Q\in \mathcal T}  g _{Q} ^p 
\Bigr] ^{1/p} 
	\Bigr\rVert_{L ^{p} (w)}  
\end{equation*}
\end{lemma} 

\begin{proof}
This is a well-known estimate on the $ A_p$ norm of simple averaging operators. 
Writing $ \sigma (x) = w (x) ^{1-p'}$,   we will exchange out an average over Lebesgue measure 
for an average over $ \sigma $-measure.  Thus, set 
\begin{equation*}
\langle \psi  \rangle ^{\sigma } _{Q} := \sigma (Q) ^{-1} \int _{Q} \psi \; d \sigma \,.
\end{equation*}
We can estimate as follows. 
\begin{align*}
\int_{\mathbb{R}^n}
 \langle   g  \rangle_{\rho Q} ^p \mathbf{1}  \;d w
&=   \Bigl( \langle   g  \sigma^{-1}  \rangle_{\rho Q} ^{\sigma} \Bigr) ^p \Bigl( \frac{\sigma(\rho Q)}{|\rho Q|} \Bigr)^p w(\rho Q) \\
&\le [w]_{A_p}   \sigma(Q) \bigl( \langle   g  \sigma^{-1}  \rangle_{\rho Q} ^{\sigma} \bigr) ^p \\
&= [w]_{A_p}  \int_{\mathbb{R}^n} g^p  \sigma^{-p}\; d \sigma= [w]_{A_p} \int_{\mathbb{R}^n}   g^p \; dw.   
\end{align*}
And the Lemma is a trivial extension of this inequality. 
\end{proof}

\subsection{The Case of $1 < p < 2$}

We let $k_\epsilon \simeq \epsilon^{-1}$ be a constant such that
\begin{align*}
w \Bigl\{
	\sum_{Q \in \mathcal S \backslash \mathcal S ^{1}}  \langle f \rangle_{\rho Q}^2 \mathbf{1}_{Q}  >  k_{\epsilon }
	\Bigr\} 
	& =
		w \Bigl\{
		\sum_{\ell =0} ^{\infty }\sum_{Q \in \mathcal S _{\ell }}  2 ^{-2\ell} \mathbf{1}_{Q}  >   \sum_{\ell =0} ^{\infty } 2 ^{- \epsilon \ell }
	\Bigr\} 
	& \le 
	\sum_{\ell =0} ^{\infty }
		w \Bigl\{
		\sum_{Q \in \mathcal S _{\ell }}  2 ^{-2\ell} \mathbf{1}_{Q}  > 2 ^{- \epsilon \ell }
	\Bigr\}.
\end{align*}

For fixed $\ell$ we may estimate, using \eqref{e:E}, and exchanging out a square for a $ p$th power, 
 \begin{align*}
w \Bigl\{
\sum_{Q \in \mathcal S _{\ell }}   2 ^{-2\ell} \mathbf{1}_{Q}  > 2 ^{- \epsilon  \ell }
	\Bigr\} 
	& \le
	w \Bigl\{
	\sum_{Q \in \mathcal S _{\ell }}  \langle f \mathbf{1} _{E (Q)} \rangle _{\rho Q}^{p} \mathbf{1}_{Q}  \gtrsim 2 ^{ (2-p-\epsilon)  \ell    }
	\Bigr\} 
	\\
	& \lesssim  [w] _{A_p}   2 ^{- (2-p-\epsilon)p/2  \ell } \lVert f\rVert_{L ^{p} (w)} ^{p}  
\end{align*}
where in the last inequality we have used Lemma \ref{l.weakrho}. Choosing $\epsilon = 1 - p/2$ and summing over $\ell$ gives the result.  

\subsection{The case of $p=2$}
Of course the estimate is a bit crude.  For a large constant $ C$, take $ \ell _0$ to be the integer part of 
$ C(1+\log_2 [ w ] _{A_2}) $.   Estimate 
 \begin{align*}
 w \Bigl\{
	\sum_{Q \in \mathcal S \backslash \mathcal S ^{1}}  \langle f \rangle_{\rho Q }^2 \mathbf{1}_{Q} > 2 
	\Bigr\} & \le
	w \Bigl\{
		\sum_{\ell = 0} ^{\ell_0 -1 } \sum_{Q \in \mathcal S _{\ell }} \langle f \rangle_{\rho Q}^2 \mathbf{1}_{Q}  >  1
	\Bigr\} 
	+
		w \Bigl\{ \sum_{\ell = \ell_0} ^{ \infty  }
		\sum_{Q \in \mathcal S _{\ell }}   \langle f \rangle_{\rho Q}^2 \mathbf{1}_{Q}  >  
		\sum_{\ell = \ell_0} ^{ \infty  } 2 ^{- \ell /8}
	\Bigr\} \\ 
&\le 	\sum_{\ell = 0} ^{\ell_0 -1 }  w \Bigl\{
		\sum_{Q \in \mathcal S _{\ell }} \langle f \rangle_{\rho Q}^2 \mathbf{1}_{Q}  >   \frac 1 {\ell _0}
	\Bigr\} 
	+
\sum_{\ell = \ell_0} ^{\infty }		w \Bigl\{ 
		\sum_{Q \in \mathcal S _{\ell }}   \langle f \rangle_{\rho Q}^2 \mathbf{1}_{Q} \gtrsim 2^{-\ell/8 }     
	\Bigr\}	
\end{align*}

Recall the  the $ A _{\infty }$ property for $ A_2 $ weights:  For any cube $ Q$ and $ E \subset Q$, with $ \lvert  E\rvert < \tfrac 12 \lvert  Q\rvert  $,  there holds 
\begin{equation*}
w (E) < \bigl(1 - \tfrac c {[w] _{A_2}} \bigr) w (Q) \,, 
\end{equation*}
for absolute choice of constant $ c$. Applying this in an inductive fashion, we see that 
 \begin{align*}
 	 w \bigl\{ \sum_{Q\in \mathcal S _{\ell }}  \langle f \rangle_{\rho Q} ^2  \mathbf 1_{Q} \gtrsim 2^{- \ell /8 } \bigr\} 
 	 & \le 
 	 w \bigl\{ \sum_{Q\in \mathcal S _{\ell }}   \mathbf 1_{Q} \gtrsim 2 ^{15\ell /8}\bigr\} 
\\
& \lesssim \operatorname {exp}(( - c 2 ^{15\ell /8}) / [w] _{A_2} ) 
w \{ \bigcup \{Q \;:\; Q\in \mathcal S _{\ell }\}\} 
\\ 
& \lesssim [w] _{A_2} 2 ^{ \ell } \operatorname {exp}( - c 2 ^{15\ell /8} / [w] _{A_2} ) \lVert f\rVert_{ L ^2 (w)} ^2 \,. 
\end{align*} 
where $0 < c < 1$ is a fixed constant. This is  summable in $ \ell\ge \ell _0 $ to at most a constant, 
for $ C$ sufficiently large. 

For the case of $ 0\le \ell < \ell _0$, we use the estimate of Lemma \ref{l.weakrho} to obtain 
\begin{align*}
	\sum_{\ell =0} ^{\ell _0-1} 	w \bigl\{\sum_{Q\in \mathcal S _{\ell } }  \langle f \mathbf{1} _{E (Q)}\rangle_{\rho Q} ^2  \mathbf 1_{Q} > 
	\frac 1 { \ell _0} 
	\bigr\} 
	& \lesssim \ell_0 ^{2} [w] _{A_2} \lVert f\rVert_{L ^2 (w)} ^2 = 
	[w] _{A_2} (1+\log [w] _{A_2}) ^2 \lVert f\rVert_{L ^2 (w)} ^2\,  
\end{align*}
concluding the proof of this case.

\subsection{The case of $2 < p $}
The case of $ p=2$ is the critical case, and so the case of $ p$ larger than $ 2$ follows from 
extrapolation. However, here we are extrapolating weak-type estimates. It is known 
that this is possible, with estimates on constants. We outline the familiar argument as found in \cite{2480568}. 

We have 
\begin{align*}
w \Bigl\{ \Bigl[ \sum_{Q\in \mathcal S}  \langle    f    \rangle_Q^2 \mathbf{1}_{Q} \Bigr] ^{1/2} > 1  \Bigr\}^{\frac{1}{p}} &= 
 \Bigl( w \Bigl\{ \Bigl[ \sum_{Q\in \mathcal S}  \langle    f    \rangle_Q^2 \mathbf{1}_{Q} \Bigr] ^{1/2} > 1  \Bigr \} ^{\frac{2}{p}} \Bigr)^{\frac{1}{2}} \\
&=  \Bigl( hw \Bigl\{ \Bigl[ \sum_{Q\in \mathcal S}  \langle    f    \rangle_Q^2 \mathbf{1}_{Q} \Bigr] ^{1/2} > 1  \Bigr \}  \Bigr)^{\frac{1}{2}} 
\end{align*}
for $h \in L^{q'}(w)$ with norm $1$, where $q = \frac{p}{2}$. Now by the Rubio de Francia algorithm there is a function $H$ such that
\begin{itemize}  
\item[i.] $h \leq H$ 
\item[ii.] $\lVert H \rVert_{L^{q'}(w)} \lesssim \lVert h \rVert_{L^{q'}(w)}$
\item[iii.] $Hw \in A_1$
\item[iv.] $[Hw]_{A_1} \lesssim [w]_{A_p}$.
\end{itemize}
We can continue,
\begin{align*}
\bigl( hw \Bigl\{ \Bigl[ \sum_{Q\in \mathcal S}  \langle    f    \rangle_Q^2 \mathbf{1}_{Q} \Bigr] ^{1/2} > 1  \Bigr \} \bigr)^{\frac{1}{2}} &\leq
\bigl( Hw \Bigl\{ \Bigl[ \sum_{Q\in \mathcal S} \langle    f    \rangle_Q^2 \mathbf{1}_{Q}  \Bigr] ^{1/2} > 1  \Bigr \} \bigr)^{\frac{1}{2}}\\ 
&\lesssim \Biggl(  [Hw]_{A_2} \bigl(1 + \log{[Hw]_{A_2}} \bigr)^2  \int_{\mathbb{R}} f^2 H w \Biggr)^{\frac{1}{2}} \\
&\lesssim [H w]_{A_2}^{\frac{1}{2}} \bigl(1 + \log{[Hw]_{A_2}} \bigr)\lVert f \rVert_{L^p(w)} \lVert H \rVert_{L^{q'}(w)} \\
&\lesssim [w]_{A_p}^{\frac{1}{2}} \bigl(1 + \log{[w]_{A_p}} \bigr)  \lVert f \rVert_{L^p(w)}. 
\end{align*}

\begin{remark}\label{r:vector} 
The estimate in Theorem~\ref{t:sparse} is a weighted estimate for a  vector-valued dyadic positive 
operator.  One of us \cite{1007.3089} has characterized such inequalities in terms of testing conditions.  
Using this condition, we could not succeed in eliminating the logarithmic estimate in the case of $ p=2$. 
It did however suggest one of the examples in the next section.  
\end{remark}

\section{Examples} 
 
The usual example of a power weight in one dimension $ w (x)= \lvert  x\rvert ^{ \epsilon-1 } $, with $0< \epsilon < 1 $ 
has $ [w] _{A _p } \simeq \epsilon ^{-1}$. It is straight forward to see that for $ \sigma (x) = w (x) ^{1-p'}$, 
and appropriate constant $ c = c (p)$, we have 
\begin{equation*}
c ^{p} w \{ S \mathbf 1_{[0,1]} > c\} = c ^{p} w ([0,1]) \simeq \epsilon ^{-1}  
\end{equation*}
whereas $ \lVert \mathbf 1_{[0,1]}\rVert_{L ^{p} (\sigma )} \simeq 1$.  Hence, the smallest power on $ [w] _{A_p}$ that we 
can have is $ \frac 1p$.  

There is a finer example, expressed through the dual inequality, which shows that the power on $ [w ] _{A_p}$ can never be less than $ \frac12$. 
 Consider the Haar square function inequality 
$ \lVert S (\sigma f)\rVert_{L ^{p, \infty } (w)} \lesssim \lVert f\rVert_{L ^{p} (\sigma )}$, where $ \sigma (x) = w (x) ^{1-p'}$ 
is the dual measure.  Viewing this as  a map from $ L ^{p} (\sigma )$ to $ L ^{p, \infty } (w; \ell ^2 )$, 
the dual map takes $ L ^{p', 1} (w; \ell ^2 )$ to $ L ^{p'} (\sigma )$, and the inequality is 
\begin{equation} \label{e.dsq}
\Bigl\lVert  \sum_{Q}  \lvert  Q\rvert ^{1/2} \langle a_Q \cdot w \rangle_Q h _{Q}  \Bigr\rVert_{L ^{p'} (\sigma )} 
\lesssim \Bigl\lVert\Bigl[ \sum_{Q} a_Q ^2  \Bigr] ^{1/2}  \Bigr\rVert_{ L ^{p',1} (w)} \,. 
\end{equation}
In the inequality, $ \{a_Q\}$ are a sequence of measurable functions.  
We show that the implied constant is at least $ C_p [w] _{A_p} ^{\beta  }$, for any $ 0< \beta < \tfrac 12 $.

In the inequality \eqref{e.dsq}, the right hand side is independent of the signs of the functions $ a_Q$. Hence it, with the 
standard Khintchine estimate, implies the   inequality 
\begin{equation*}
\Bigl\lVert  \Bigl[\sum_{Q}   \langle a_Q \cdot w \rangle_Q ^2  \mathbf 1_{Q} \Bigr] ^{1/2}   \Bigr\rVert_{L ^{p'} (\sigma )} 
\lesssim \Bigl\lVert\Bigl[ \sum_{Q} a_Q ^2  \Bigr] ^{1/2}  \Bigr\rVert_{ L ^{p',1} (w)} 
\end{equation*}
As the left hand side is purely positive, we prefer this form.  
Indeed, we specialize the inequality above to one which is of testing form.
Take the functions $ \{a_k\}$ to be 
\begin{equation*}
a _{k} (x) :=c  \sum_{j=k+1} ^{\infty } \frac 1 { (j-k) ^{ \alpha }  } \mathbf 1_{[ 2 ^{-j}, 2 ^{-j+1})}\,, \qquad \tfrac 12 < \alpha < 1\,, \ k\in \mathbb N \,. 
\end{equation*}
For appropriate choice of constant $ c= c _{\alpha }$, there holds $ \sum_{k=1} ^{\infty } a_k (x) ^2 \le  \mathbf 1_{[0,1]} $, 
whence we have 
\begin{equation*}
\Bigl\lVert  \Bigl[\sum_{k=1} ^{\infty } a_k (x) ^2  \Bigr] ^{1/2}  \Bigr\rVert_{L ^{p',1} (w)} \lesssim w ([0,1]) ^{1/p'} \,. 
\end{equation*}

The next and critical point, concerns the terms $  \langle a_k \cdot w \rangle _{[0, 2 ^{-k})}$. 
Recalling $  w (x)= \lvert  x\rvert ^{\epsilon-1 } $, we have 	
\begin{align*}
 \langle a_k \cdot w \rangle _{[0, 2 ^{-k})} &\simeq  2 ^{k} 
 \sum_{j=k+1} ^{\infty } 
  \frac 1 { (j-k) ^{\alpha }  }   2 ^{- \epsilon j} 
  \\
  &= 2 ^{k (1- \epsilon )}  \sum_{j=1} ^{\infty } 
  \frac 1 { j ^{\alpha }  }  2 ^{- \epsilon j} 
 \\
 & \simeq 2 ^{k (1- \epsilon )} \int _{1} ^{\infty } \frac 1 {x ^{\alpha }  } 2 ^{- \epsilon x} \; dx \simeq \epsilon ^{-1+ \alpha } 2 ^{k (1- \epsilon )} \,. 
\end{align*}
From this we conclude that the testing term is 
\begin{align*}
\int _{[0,1]} \Bigl[ \sum_{k=1} ^{\infty }  \langle a_k \cdot w \rangle _{[0, 2 ^{-k})} ^2   \mathbf 1_{[0, 2 ^{-k})} \Bigr]^{p'/2}  \; d \sigma  
& \simeq  \epsilon ^{(-1+ \alpha) p' }  
\int _{[0,1]} w (x) ^{p'} \; d \sigma (x)
\\
& \simeq   \epsilon ^{(-1+ \alpha) p'  }  
\int _{[0,1]} w (x) ^{p'}  w (x) ^{1-p'} \; dx  
\\&
\simeq   [w] _{A_p} ^{(1 - \alpha) p' } w ([0,1]) \,.   
\end{align*}
Therefore, the power on $ [w] _{A_p}$ in the implied constant in \eqref{e.dsq} can never be strictly less than $ \frac 12$, 
since  $ \alpha > 1/2 $ is arbitrary.

\end{document}